\documentclass[letterpaper, 10 pt, conference]{ieeeconf}  

\IEEEoverridecommandlockouts                              

\usepackage{graphics} 
\usepackage{subfig}
\usepackage{epsfig} 
\usepackage{mathptmx} 
\usepackage{times} 
\usepackage{amsmath} 
\usepackage{amssymb}  
\usepackage{multirow}
\usepackage{algorithm}
\usepackage{algorithmic}
\usepackage{paralist}

\newtheorem{theorem}{Theorem}

\newtheorem{lemma}[theorem]{Lemma}

\newtheorem{definition}{Definition}

\newtheorem{remark}[theorem]{Remark}
\newtheorem{assumption}[theorem]{Assumption}

\newcommand{\BB}{\mathbb{B}}

\,
\,

\newcommand{\w}{\mathbf w}
\newcommand{\x}{\mathbf x}
\newcommand{\y}{\mathbf y}
\newcommand{\z}{\mathbf z}

\newcommand{\argmin}{\mathop{\rm argmin}}

\newcommand{\prox}{\textnormal{prox}}

\newcommand{\br}{\mathbb{R}}

\newcommand{\ba}{\begin{array}}
\newcommand{\ea}{\end{array}}

\title{\LARGE \bf New Proximal Newton-Type Methods for Convex Optimization}

\author{Ilan Adler, Zhiyue T. Hu and Tianyi Lin\thanks{Ilan Adler and Tianyi Lin are with Department of Industrial Engineering and Operations Research, UC Berkeley. Zhiyue T. Hu is with Division of Biostatistics, UC Berkeley. Email: {\tt\small \{ilan, zyhu95, darren{\_}lin\}@berkeley.edu}}
}

\begin{document}

\maketitle
\thispagestyle{empty}
\pagestyle{empty}

\begin{abstract}
In this paper, we propose new proximal Newton-type methods for convex optimization problems in composite form. The applications include model predictive control (MPC) and embedded MPC. Our new methods are computationally attractive since they do not require evaluating the Hessian at each iteration while keeping fast convergence rate. More specifically, we prove the global convergence is guaranteed and the superlinear convergence is achieved in the vicinity of an optimal solution. We also develop several practical variants by incorporating quasi-Newton and inexact subproblem solving schemes and provide theoretical guarantee for them under certain conditions. Experimental results on real-world datasets demonstrate the effectiveness and efficiency of new methods.
\end{abstract}

\section{Introduction}\label{sec:introduction}
We consider the generic convex composite optimization model: 
\begin{equation}\label{prob:main}
\min_{\x \in \br^d} \ F(\x) \ = \ f(\x) + r(\x),
\end{equation}
where $f$ is a convex and \textit{twice differentiable} function and $r$ is an extended real-valued closed convex function. Problem~\eqref{prob:main} have found various applications ranging from model predictive control (MPC)~\cite{Patrinos-2011-Global, Nedelcu-2012-Iteration, Patrinos-2013-Accelerated} to machine learning and statistics~\cite{Boyd-2004-Convex, Hastie-2009-Elements}. For example, when $f$ is quadratic and $r$ is an indicator of a polyhedral set, problem~\eqref{prob:main} becomes a quadratic program (QP) and covers numerous applications in embedded MPC. When $f = (1/n)\sum_{i=1}^n \phi(\cdot, \w_i, \y_i)$ for $n$ data samples $\{(\w_i, y_i)\}_{i=1}^n$ and $r = \|\cdot\|_1$, problem~\eqref{prob:main} reduces an empirical risk minimization in sparse learning, e.g., $\ell_1$-regularized logistic and Poisson regression; see~\cite{Hastie-2015-Statistical}. 

During the past decades, many optimization algorithms have been developed for solving problem~\eqref{prob:main} with theoretical guarantees; see proximal splitting method~\cite{Combettes-2011-Proximal} and its acceleration~\cite{Tseng-2008-Accelerated, Beck-2009-Fast, Nesterov-2013-Gradient}. These methods also perform admirably in practice and have been implemented in the TFCOS package \cite{Becker-2011-Templates}. Comparing to their first-order counterpart, second-order methods for convex optimization enjoy superior convergence in both theory and practice. For example, the proximal Newton-type methods~\cite{Becker-2012-Quasi, Patrinos-2013-Proximal, Lee-2014-Proximal} achieve at least a superlinear convergence rate while first-order method only achieves much slower $O(1/k^2)$ rate which is known to be unimprovable~\cite{Nesterov-2018-Lectures}. Furthermore, second-order methods are more robust and depend less on the problem structure than first-order methods, which often suffer from the tuning of step size on ill-conditioned problems. These advantages are due to the curvature exploration of second-order methods and are further demonstrated by the admirable performance on several specific problems, include GLMNET~\cite{Friedman-2007-Pathwise}, QUIC~\cite{Hsieh-2014-Quic} and PNOPT~\cite{Lee-2014-Proximal}. However, each step of Newton-type methods requires solving a composite QP defined by the Hessian matrix, which poses a tremendous numerical challenge for big data applications. Despite some recent stochastic Newton-type methods based on subsampling and sketching techniques~\cite{Pilanci-2017-Newton, Roosta-2019-Sub}, the issue on curse of dimensionality remains. Thus, it is natural to ask:
\begin{quote}
\textit{Can we balance the curvature exploration and per-iteration cost in second-order methods?}
\end{quote}
In this paper, we present an affirmative answer by developing new proximal Newton-type methods. These methods are inspired by Shamanskii's seminal work~\cite{Shamanskii-1967-Modification}. They are computationally attractive since they do not require evaluating the Hessian at each iteration while keeping fast convergence rate. The proof technique is new and of independent interest. 

\textbf{Related works:} Problem~\eqref{prob:main} is equivalent to finding a simple root of a multivariate nonlinear equation $\nabla f(\x) = 0$ when $r = 0$, where Newton-type methods serve as the state-of-the-art approach~\cite{Kelley-1995-Iterative, Kelley-2003-Solving}. Recent works focuses on the development of Newton-type methods with a superquadratic rate of convergence~\cite{Homeier-2004-Modified, Frontini-2004-Third, Cordero-2007-Variants, Noor-2007-Third, Darvishi-2007-Third, Hueso-2009-Third, Waseem-2016-Efficient}. Despite the appealing local convergence property, these methods require the Lipschitz continuity of high-order derivatives of $f$ and suffer from the expensive per-iteration computational cost of forming and factorizing a new Hessian matrix at least once at each iteration. On the other hand, when $r=0$, our method with unit stepsize and exact Hessian information reduces to Shamanskii's method~\cite{Shamanskii-1967-Modification} whose global convergence has been studied in~\cite{Lampariello-2001-Global}. Comparing to Shamanskii's method, our algorithmic scheme is more flexible and the convergence results are more general and comprehensive; see Section~\ref{sec:algorithm} and~\ref{sec:convergence} for details.

\textbf{Organization:} The rest of the paper is organized as follows. In Section~\ref{sec:prelim}, we introduce notations and assumptions. In Section~\ref{sec:algorithm}, we present new methods, namely generic and inexact proximal Shamanskii methods, for solving problem~\eqref{prob:main}. We present the convergence guarantee and empirical results on two real-world datasets in Section~\ref{sec:convergence} and~\ref{sec:results}. We conclude in Section~\ref{sec:conclusions} and defer the proofs to the appendix. 

\section{Preliminaries}\label{sec:prelim}
\textbf{Notation:} We denote vectors and matrices by bold lower and regular upper case letters. $\left\|\cdot\right\|$ denotes the $\ell_2$-norm and the matrix spectral norm. $\top$ denotes the transpose of a vector. $\BB_\delta(\x) = \{\z \in \br^d \mid \|\z - \x\| \leq \delta\}$ refers to $\delta$-neighborhood of $\x$. Let $A$ and $B$ be two symmetric matrices, $A \succeq B$ indicates that $A-B$ is positive semidefinite. $I$ is an identity matrix. The induced norm from $H \succeq 0$ is $\|\x\|_H := \sqrt{\x^\top H\x}$. $\prox_g$ refers to the proximal mapping of a convex function $g$; see~\cite{Parikh-2014-Proximal}. The notation $a_t = o(b_t)$ means $a_t/b_t \rightarrow 0$ as $t \rightarrow +\infty$.   

\textbf{Objectives in convex composite optimization:} We wish to find a point that globally minimizes the objective $F$.
\begin{definition}
$\x^* \in \br^d$ is an optimal solution set to problem~\eqref{prob:main} if $F(\x) - F(\x^*) \geq 0$ for all $\x \in \br^d$. 
\end{definition}

In general, finding one global optimal solution is NP-hard~\cite{Murty-1987-Some} but standard for convex composite optimization.
\begin{assumption}\label{Assumption:objective-convex-main}
$f$ and $r$ are both convex:
\begin{equation*}
\begin{array}{rl}
f(\y) - f(\x) - (\y - \x)^\top\nabla f(\x) \ \geq \ 0, & \forall \x, \y \in \br^d, \\
r(\y) - r(\x) - (\y - \x)^\top\xi \ \geq \ 0, & \forall \x, \y \in \br^d. 
\end{array}
\end{equation*}
where $\xi \in \partial r(\x)$ is a subgradient of $r$ at $\x$. 
\end{assumption}

Assumption~\ref{Assumption:objective-convex-main} makes the convergence of algorithms to $\x^*$ computationally feasible. For the Newton-type methods, the convergence property depends on the scaled proximal mapping~\cite{Patrinos-2013-Proximal, Lee-2014-Proximal} of $r$ and the gradient and Hessian of $f$. Thus, it is necessary to impose additional conditions on $f$ and $r$. A minimal set of conditions that have become standard in the literature~\cite{Boyd-2004-Convex, Patrinos-2013-Proximal, Lee-2014-Proximal, Nesterov-2018-Lectures} are as follows: 
\begin{assumption}\label{Assumption:smooth-main}
$f$ is $\ell$-gradient and $\rho$-Hessian Lipschitz: $\|\nabla f(\x) - \nabla f(\y)\| \leq \ell\|\x - \y\|$ and $\|\nabla^2 f(\x) - \nabla^2 f(\y)\| \leq \rho\|\x - \y\|$ for all $\x, \y \in \br^d$.
\end{assumption}
\begin{assumption}\label{Assumption:proximal-main}
The scaled proximal mapping of $r$ with a matrix $H \succ 0$, i.e., $\prox_r^H(\x) := \argmin \ \{r(\z) + \frac{1}{2}\|\z - \x\|_H^2\}$, can be efficiently computed for any $\x \in \br^d$.  
\end{assumption}

Assumption~\ref{Assumption:smooth-main} is satisfied in many applications. $f$ is smooth and $r$ is an indicator function of a convex and bounded set~\cite{Parikh-2014-Proximal}. Assumption~\ref{Assumption:proximal-main} is also not restrictive since $\prox_r^H(\x)$ can be efficiently computed using subgradient method~\cite{Boyd-2004-Convex} in general. In sparse learning when $r$ is $\ell_1$-norm, even faster accelerated projected gradient method~\cite{Beck-2009-Fast} can be applicable. In addition, some stopping criteria for approximating $\prox_r^H(\x)$ are discussed in~\cite{Lee-2014-Proximal}. On the other hand, there is rich curvature information of $f$ around $\x^*$ which stands for fast local convergence of the algorithms~\cite{Roosta-2019-Sub}. Thus, we make the following assumption. 
\begin{assumption}\label{Assumption:regular-main}
$\nabla^2 f(\x^*)$ is invertible. 
\end{assumption}

Throughout this paper, the algorithm efficiency is quantified by the order of convergence to $\x^*$. As an example, the order of the convergence of the proximal Newton method is at least two under certain conditions~\cite{Lee-2014-Proximal}. Formally,
\begin{definition}\label{Def:order-convergence}
Letting $\{\x_t\}_{t\geq 0}$ be the iterates generated by an algorithm and $r_t = C\|\x_t - \x^*\|$ for some constants $C > 0$, the order of convergence is $n$ if $r_t \leq r_0^{a_t}$ for all $t \geq 0$ and $\log(a_t)/t \rightarrow \log(n)$ as $t \rightarrow +\infty$.  
\end{definition}

With these definitions in mind, we ask if Newton-type methods can achieve favorable per-iteration cost while not sacrificing the order of convergence too much. 

\section{Algorithm}\label{sec:algorithm}
In this section, we present two new methods for solving~\eqref{prob:main}. These two methods can be interpreted as exact and inexact proximal extension of the Shamanskii method~\cite{Shamanskii-1967-Modification} for convex optimization in composite form. 

\subsection{Generic proximal Shamanskii method}
We present some basic ideas behind the algorithmic design of generic proximal Shamanskii method. Recalling that each iteration of proximal Newton method is defined by minimizing the sum of the second-order Taylor expansion of $f$ at $\x_t$ and $r$ as follows:   
\begin{equation*}
\min_{\x \in \br^d} \ (\x-\x_t)^\top\nabla f(\x_t) + \frac{1}{2}(\x-\x_t)^\top\nabla^2 f(\x_t)(\x-\x_t) + r(\x). 
\end{equation*}
This is called the proximal Newton step and can be solved by iterative solvers where Hessian-vector products dominate the cost. The computational cost is high when the Hessian is dense and of high dimension. In Algorithm~\ref{Algorithm:PSA}, we update $H_t$ using a positive-definite approximation to the Hessian $\nabla^2 f(\x_t)$ if $t \ \text{mod} \ n = 0$ and conduct the backtracking line search~\cite{Boyd-2004-Convex} to select a suitable stepsize $\alpha > 0$ at each iteration. We let $\Delta\x_t =\x' - \x_t$ and describe the sufficient descent criterion for backtracking line search as follows:  
\begin{equation*}\small
F(\x_t + \alpha\Delta\x_t) - F(\x_t) \leq (\alpha/4)((\Delta\x_t)^\top\nabla f(\x_t) + r(\x') - r(\x_t)). 
\end{equation*}
For the next $n-1$ iterations, we compute a proximal Newton step with possibly the delayed Hessian $H_t$. 

Furthermore, this approach pursues a good balance between the order of convergence and per-iteration cost. Indeed, we can find the desired sparse matrix $R$ using sparse Cholesky factorization~\cite{Demmel-1997-Applied, Chen-2008-Algorithm} if $t \ \text{mod} \ n = 0$ and compute proximal Newton steps with $R$ instead of $H_t$; see Algorithm~\ref{Algorithm:PSA}. Since $R$ is sparse, the per-iteration cost can be much cheaper than proximal Newton method. On the other hand, we show that the order of convergence is $\sqrt[n]{n+1}$ which does not deteriorate too much if we choose $n \geq 1$ properly. 

Finally, our method covers a few classical methods as special cases. It becomes proximal Newton method when $n=1$ and the chord method~\cite{Kelley-2003-Solving} when $n = \infty$ and $r = 0$. 
\begin{algorithm}[!t]
\caption{A generic proximal Shamanskii method}\label{Algorithm:PSA}
\begin{algorithmic}
\STATE \textbf{Input:} $n \geq 1$ and $\x_0\in\br^d$. 
\FOR{$t=0, 1, 2, \ldots, T-1$}
\IF{$t \ \text{mod} \ n = 0$}
\STATE Update $H_t$ using a positive-definite approximation to the Hessian $\nabla^2 f(\x_t)$.
\STATE Find $R \in \br^{d \times d}$ so that $H_t = RR^\top$ and $R$ is sparse.  
\ENDIF
\STATE $\x' = \argmin \{r(\z) + (\z-\x_t)^\top\nabla f(\x_t) + \frac{1}{2}\|R(\z - \x_t)\|^2\}$. 
\STATE Update $\alpha \in (0, 1]$ using backtracking line search. 
\STATE $\Delta\x_t =\x' - \x_t$ and $\x_{t+1} = \x_t + \alpha\Delta\x_t$. 
\ENDFOR
\STATE \textbf{Output:} $\x_T$. 
\end{algorithmic}
\end{algorithm}
\subsection{Inexact proximal Shamanskii method}
We propose inexact proximal Shamanskii method which generalizes inexact Newton method~\cite{Dembo-1982-Inexact, Eisenstat-1996-Choosing}. This approach is crucial in practice since it is impossible to perform an exact proximal Newton step in general. For example, when $r$ is $\ell_1$-norm, Assumption~\ref{Assumption:proximal-main} is satisfied but an exact proximal Newton step does not have a closed-form solution. Since we are interested in the local behavior, we assume that $\x_t$ is sufficiently close to $\x^*$ and $\alpha=1$ which are made in~\cite{Dembo-1982-Inexact, Eisenstat-1996-Choosing} for analyzing inexact Newton method and~\cite{Lee-2014-Proximal} for analyzing inexact proximal Newton method. The remaining practical concern is how inexactly we perform a proximal Newton step is critical to the performance of the method. 

We first define two key notions: $G(\x) = \ell(\x - \prox_{r/\ell}(\x - \nabla f(\x)/\ell))$ and $\widehat{G}(\x', \x, H) = \ell(\x' - \prox_{r/\ell}(\x' - (\nabla f(\x) + H(\x' - \x))/\ell))$. The first one is called \textit{composite gradient} which generalizes gradient to convex composite optimization~\cite{Nesterov-2018-Lectures} and the second one is the composite gradient for each proximal Newton step and resorts to measure the near-stationarity of $\x'$ for solving the proximal Newton step. $\|G(\x_t)\|$ refers to the near-stationarity of $\x_t$ for solving problem~\eqref{prob:main} and $\|\widehat{G}(\x_{t+1}, \x_t, H)\|$ characterizes the extent of the exactness of the proximal Newton step. Moreover, computing $G(\x)$ and $\widehat{G}(\x', \x, H)$ is relatively cheap since the proximal mapping $\prox_{r/\ell}$ has the closed-form solution for many commonly used functions $r$, e.g., $\ell_1$-norm.   

Now we can define the stopping criterion with $\eta_t > 0$ and $\gamma \geq 1$ as follows: 
\begin{equation}\label{criterion:stop-rule}
\|\widehat{G}(\x_{t+1}, \x_t, H_t)\| \ \leq \ \eta_t\|G(\x_t)\|^\gamma. 
\end{equation}
which implies that we do not need to solve the proximal Newton step very accurately when the iterate $\x_t$ is far from the optimal solution $\x^*$, i.e., $\|G(\x_t)\|$ is large. 
\begin{algorithm}[!t]
\caption{Inexact proximal Shamanskii method}\label{Algorithm:IPSA}
\begin{algorithmic}
\STATE \textbf{Input:} $n \geq 1$, $\x_0\in\br^d$, $\{\eta_t\}_{t \geq 0}$ and $\gamma \geq 1$. 
\FOR{$t=0, 1, 2, \ldots, T-1$}
\IF{$t \ \text{mod} \ n = 0$}
\STATE Find $R \in \br^{d \times d}$ so that $\nabla^2 f(\x_t) = RR^\top$ and $R$ is sparse. 
\ENDIF
\STATE $\x_{t+1} \approx \argmin \{r(\z) + (\z-\x_t)^\top\nabla f(\x_t) + \frac{1}{2}\|R(\z - \x_t)\|^2\}$ such that the stopping criterion~\eqref{criterion:stop-rule} is satisfied. 
\ENDFOR
\end{algorithmic}
\end{algorithm}
\section{Convergence results}\label{sec:convergence}
We first provide the global convergence guarantee for Algorithm~\ref{Algorithm:PSA}. Even if there are many similar results for Newton methods, e.g.,~\cite[Section 4]{Patriksson-2013-Nonlinear}, our result is the first global convergence for general Shamanskii-type methods to our knowledge.
\begin{theorem}\label{theorem:generic-global}
Under Assumption~\ref{Assumption:objective-convex-main}-\ref{Assumption:proximal-main} and $H_t \succeq m I$ for some $m > 0$, the iterate $\{\x_t\}_{t\geq 0}$ generated by Algorithm~\ref{Algorithm:PSA} satisfies that $\|\Delta\x_t\| \rightarrow 0$ and $\x_t$ converges to $\x^*$ as $t \rightarrow +\infty$. 
\end{theorem}
\begin{remark}
Algorithm~\ref{Algorithm:PSA} behave like a first-order method if $H_t$ does not approximate the Hessian $\nabla^2 f(\x_t)$ well. In this case, even when the iterate $\x_t$ is very close to $\x^*$, the local rate will be at most linear in general and $\alpha < 1$. 
\end{remark}

We proceed to the local convergence of Algorithm~\ref{Algorithm:PSA}. The first theorem focuses on the case when $H_t = \nabla^2 f(\x_t)$. 
\begin{theorem}\label{Theorem:local-exact-main}
Under Assumption~\ref{Assumption:objective-convex-main}-\ref{Assumption:regular-main} and let $\{\x_t\}_{t\geq 0}$ be generated by Algorithm~\ref{Algorithm:PSA}, $\alpha=1$ is satisfied by sufficient descent criterion for some large $t$. There also exists $\delta, m >0$ such that, if we let $r_t = (3\rho/2m)\|\x_t - \x^*\|$ and set $\|\x_0 - \x^*\| \leq \{\delta, 2m/3\rho\}$, then $r_t \leq 1$ for all $t \geq 0$ and $r_t \leq r_0^{(n+1)^k(l+1)}$ for $t = nk+l$ where $k \geq 0$ and $0 \leq l \leq n-1$ are integers. 
\end{theorem}
\begin{remark}\label{Remark:local-exact-main}
Using Theorem~\ref{Theorem:local-exact-main}, we have $a_t = (n+1)^k(l+1)$ for $t = nk+l$ in Definition~\ref{Def:order-convergence} and 
\begin{equation*}
\lim_{t \rightarrow +\infty} \frac{\log(a_t)}{t} = \lim_{k \rightarrow +\infty} \frac{k\log(n+1) + \log(l+1)}{nk+l} = \log(\sqrt[n]{n+1}). 
\end{equation*}
This implies that the order of convergence is at least $\sqrt[n]{n+1}$. 
\end{remark}
\begin{remark}
When $r=0$ and $f$ satisfies further regularity conditions, Theorem~\ref{Theorem:local-exact-main} can be derived using the results in~\cite{Shamanskii-1967-Modification}. Compared to their proof, our techniques are much simpler and can tackle the case when $r \neq 0$. 
\end{remark}

The second theorem focuses on the case when $H_t \approx \nabla^2 f(\x_t)$ such that $\|H_t - \nabla^2 f(\x_t)\| = o(1)$ holds true. 
\begin{theorem}\label{Theorem:local-quasi-main}
Under Assumption~\ref{Assumption:objective-convex-main}-\ref{Assumption:regular-main} and let $H_t \succeq m I$ for some $m > 0$ and satisfies $\|H_t - \nabla^2 f(\x_t)\| = o(1)$. If $\x_0$ is sufficiently close to $\x^*$, then the iterates $\{\x_t\}_{t\geq 0}$ generated by Algorithm~\ref{Algorithm:PSA} achieves the superlinear convergence. 
\end{theorem}
\begin{remark}
The approximation condition was recognized as standard for analyzing the local convergence of Newton-type method with inexact Hessian~\cite{Pilanci-2017-Newton, Kohler-2017-Sub, Roosta-2019-Sub}. In practice, we observe that it can be satisfied using some modified quasi-Newton update with first-order gradients.
\end{remark}

Finally, we provide the local convergence guarantee of Algorithm~\ref{Algorithm:IPSA}. For simplicity, we denote $\theta_1=(6\rho + 2m\ell)/m$, $\theta_2 = \sqrt[\gamma-1]{(3\rho/2m)^{\gamma-1} + (\ell/2)^{\gamma-1}}$ and $\theta_3 = 3\rho/2m + \ell/2$. 
\begin{theorem}\label{Theorem:local-inexact-main}
Under Assumption~\ref{Assumption:objective-convex-main}-\ref{Assumption:regular-main} and let $\{\x_t\}_{t\geq 0}$ be generated by Algorithm~\ref{Algorithm:IPSA} with $0 < \eta_t \leq \bar{\eta} < m/16\ell$ and
\begin{equation}\label{def:local-inexact-main}
r_t \ = \ \left\{\begin{array}{cl}
\theta_1\|\x_t - \x^*\|, & \gamma = 1, \\
\theta_2\|\x_t - \x^*\|, & \gamma \in (1, 2), \\
\theta_3\|\x_t - \x^*\|, & \gamma > 2. 
\end{array}\right. 
\end{equation}
for all $t \geq 0$. If $\x_0$ satisfies that 
\begin{equation}\label{condition:local-inexact-main}
\|\x_0 - \x^*\| \leq \left\{
\begin{array}{cl}
\min\{\delta, 1/\theta_1\}, & \gamma = 1, \\
\min\{\delta, 1/\theta_1, 1/\theta_2\}, & \gamma \in (1, 2), \\
\min\{\delta, 1/\theta_1, 1/\theta_3\}, & \gamma > 2. 
\end{array}
\right. 
\end{equation}
Then $r_t \leq 1$ for all $t \geq 0$ and the following statement holds: 
\begin{equation*}
\left\{\begin{array}{lcl}
r_t \leq r_{t-1}/2 & & \gamma=1 \ \text{and} \ \eta_t \nrightarrow 0, \\
r_t = o(r_{t-1}) & & \gamma=1 \ \text{and} \ \eta_t \rightarrow 0, \\
r_t \leq r_0^{(l(\gamma-1)+1)(n(\gamma-1)+1)^k} & \text{If} & \gamma \in (1, 2) \ \text{and} \ \eta_t \nrightarrow 0, \\
r_t = o(r_0^{(l(\gamma-1)+1)(n(\gamma-1)+1)^k}) & & \gamma \in (1, 2) \ \text{and} \ \eta_t \rightarrow 0, \\
r_t \leq r_0^{(n+1)^k(l+1)} & & \gamma \geq 2. 
\end{array}\right.
\end{equation*}
for $t = nk+l$ where $k \geq 0$ and $0 \leq l \leq n-1$ are integers. 
\end{theorem}
\begin{remark}
Using Theorem~\ref{Theorem:local-exact-main}, we derive that the order of convergence is at least 
\begin{equation*}
\left\{\begin{array}{lcl}
1 & & \gamma=1 \ \text{and} \ \eta_t \nrightarrow 0, \\
> 1 & & \gamma=1 \ \text{and} \ \eta_t \rightarrow 0, \\
\sqrt[n]{n(\gamma-1)+1} & \text{If} & \gamma \in (1, 2) \ \text{and} \ \eta_t \nrightarrow 0, \\
> \sqrt[n]{n(\gamma-1)+1} & & \gamma \in (1, 2) \ \text{and} \ \eta_t \rightarrow 0, \\
\sqrt[n]{n+1} & & \gamma \geq 2. 
\end{array}\right.
\end{equation*}
The first two settings are trivial and the last setting is the same as that analyzed in Remark~\ref{Remark:local-exact-main}. For the third setting, $a_t = (l(\gamma-1)+1)(n(\gamma-1)+1)^k$ for $t = nk+l$ in Definition~\ref{Def:order-convergence} and we have
\begin{align*}
\lim_{t \rightarrow +\infty} \frac{\log(a_t)}{t} & = \lim_{k \rightarrow +\infty} \frac{k\log(n(\gamma-1)+1) + \log(l(\gamma-1)+1)}{nk+l} \\ 
& = \log(\sqrt[n]{n(\gamma-1)+1}). 
\end{align*}
This implies the desired result on the order of convergence in the third and fourth settings. 
\end{remark}
Before proceeding to the empirical part, we provide a practical approach to choose $\eta_t > 0$. The similar strategy has been designed for an inexact Newton algorithm before; see~\cite{Eisenstat-1996-Choosing}. 
\begin{theorem}\label{Theorem:inexact-step-main}
Assumption~\ref{Assumption:objective-convex-main}-\ref{Assumption:regular-main} and let $\{\x_t\}_{t\geq 0}$ be generated by Algorithm~\ref{Algorithm:IPSA} with
\begin{equation}\label{criterion:stepsize-rule}
\eta_t = \min\{\bar{\eta}, \|\widehat{G}(\x_t, \x_{t-1}, H_t) - G(\x_t)\|/\|G(\x_{t-1})\|\},  
\end{equation}
where $\bar{\eta} \in (0, m/(16\ell))$ and $H \in \br^{d \times d}$ is used for updating $\x_{t+1}$. If $\x_0$ satisfies Eq.~\eqref{condition:local-inexact-main}, then $\eta_t \rightarrow 0$ as $t \rightarrow +\infty$. 
\end{theorem}
\section{Experiments}\label{sec:results}
We present some empirical results on $\ell_1$-regularized Poisson and logistic regression problems with two real-world LIBSVM datasets\footnote{https://www.csie.ntu.edu.tw/$\sim$cjlin/libsvmtools/datasets}: \textsf{mnist} and \textsf{gisette}. The former one contains 70000 instances and 780 features, and the latter one has 7000 instances and 5000 features. Note these regression problems have been widely used for evaluating proximal Newton method~\cite{Lee-2014-Proximal} and proven in~\cite{Roosta-2019-Sub} to empirically satisfy Assumption~\ref{Assumption:objective-convex-main},~\ref{Assumption:smooth-main} and~\ref{Assumption:regular-main}. Assumption~\ref{Assumption:proximal-main} is also satisfied since the scaled proximal mapping of $\ell_1$-norm can be efficiently tackled by iterative solvers in the TFOCS package\footnote{http://cvxr.com/tfocs/}. 

\subsection{$\ell_1$-regularized Poisson regression}
We explore the effect of inexact search directions on the practical performance of Algorithm~\ref{Algorithm:IPSA} using $\ell_1$-regularized Poisson regression and the dataset \textsf{mnist}. In particular, the optimization model is 
\begin{equation}\label{prob:poisson-regression}
\min_{\x\in\br^d} \ \frac{1}{n}\sum_{i=1}^n (e^{\w_i^\top\x} - y_i\w_i^\top\x) + \mu\|\x\|_1. 
\end{equation}
where $\{(\w_i, y_i)\}_{i=1}^n$ are data samples with integer label and $\mu>0$ is chosen by five-fold cross validation. We select spectral gradient algorithm (SpaRSA)~\cite{Wright-2009-Sparse} to solve the subproblem and evaluate different stopping rules as follows,  
\begin{compactenum}
\item Solve the subproblem to high accuracy $10^{-4}$.  
\item Eq.~\eqref{criterion:stop-rule} with $\gamma=2$ and $\eta_t$ chosen in Eq.~\eqref{criterion:stepsize-rule}.
\item Eq.~\eqref{criterion:stop-rule} with $\gamma=1$ and $\eta_t$ chosen in Eq.~\eqref{criterion:stepsize-rule}.
\item Solve the subproblem with 5 maximum iterations. 
\end{compactenum}
Figure~\ref{fig:l1-possion} shows the performance of all methods on \textsf{covetype} and \textsf{mnist}. Algorithm~\ref{Algorithm:IPSA} with SpaRSA achieves the superlinear convergence under the first three stopping rules while behaving badly with 5 maximum iterations. This implies that the proper choice of $\gamma$ avoids the subproblem undersolving. The choice of $\gamma=1$ yields the fastest convergence in terms of time which is consistent with~\cite{Dembo-1982-Inexact, Lee-2014-Proximal} that the subproblem oversolving is impractical despite the theoretical guarantee. 
\begin{figure}[!t]
\centering
\includegraphics[scale=0.38]{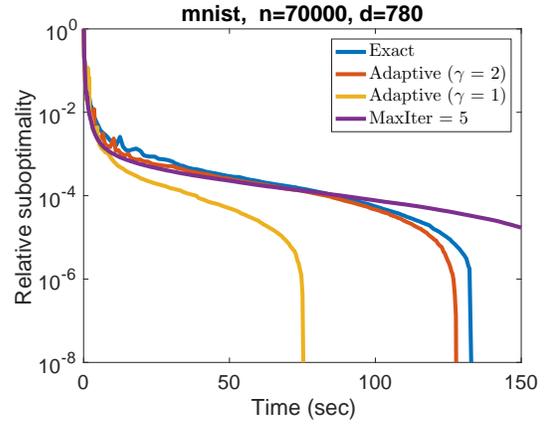}
\vspace*{-.5em}\caption{Performance of different stopping criteria on $\ell_1$-regularized Poisson regression and the dataset \textsf{mnist}.}\vspace*{-1em}\label{fig:l1-possion}
\end{figure}
\begin{figure}[!t]
\centering
\includegraphics[scale=0.38]{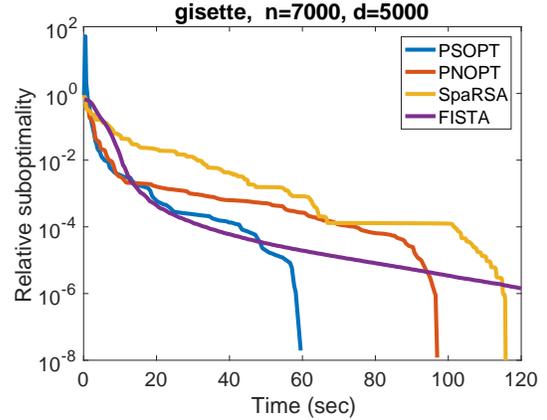}
\vspace*{-.5em}\caption{Performance of all methods on $\ell_1$-regularized logistic regression and the dataset \textsf{gisette}.}\vspace*{-2em}\label{fig:l1-logistic}
\end{figure}
\subsection{$\ell_1$-regularized logistic regression}
We compare our method, denoted by PSOPT, with other competing methods using $\ell_1$-regularized logistic regression. In particular, the optimization model is  
\begin{equation}\label{prob:logistic-regression}
\min_{\x\in\br^d} \ \frac{1}{n}\sum_{i=1}^n \log\left(1+e^{-y_i \w_i^\top\x}\right) + \mu\left\|\x\right\|_1. 
\end{equation}
where $\{(\w_i, y_i)\}_{i=1}^n$ are data samples with binary label and $\mu>0$ is chosen by five-fold cross validation. We apply the heuristics in~\cite{Lee-2014-Proximal} by constructing $H_t$ with L-BFGS update and solving the subproblem inexactly with the TFOCS package. The baseline methods include two first-order methods in the TFCOS package, i.e., SpaRSA and FISTA, and the default method in the PNOPT package\footnote{https://web.stanford.edu/group/SOL/software/pnopt/}. For our method, we set $n=3$ and the memory size $L=50$. For other methods, we use the default setting in PNOPT and TFCOS packages. 

Figure~\ref{fig:l1-logistic} shows the performance of all methods on \textsf{gisette}. Both our method and PNOPT reach high-accurate solution with much faster convergence rate than SpaRSA and FISTA and our method performs the best in terms of the relative suboptimality v.s. time. This demonstrates that our method attains fast convergence rate while keeping relatively cheap per-iteration cost and serves as a better candidate solution than standard proximal Newton-type methods sometimes. 

\section{Conclusions}\label{sec:conclusions}
Compared to the existing proximal Newton methods, our newly proposed methods are better generalizations of first-order methods that account for the curvature information while not sacrificing per-iteration computational cost too much. Experiments on real-world datasets demonstrate their effectiveness and efficiency. Future work includes studying new interior-point methods with better complexity bound.
\section*{APPENDIX}
\subsection{Proof of Theorem~\ref{theorem:generic-global}}
\begin{lemma} \label{lemma:generic-descent}
Under Assumption~\ref{Assumption:objective-convex-main}-\ref{Assumption:proximal-main}, the sufficient descent criterion is satisfied for some $\alpha \in (0, \min\{1, 3m/2\ell\})$. It also holds for all $t \geq 0$ that $F(\x_{t+1}) - F(\x_t) \leq -(m\alpha/4)\|\Delta\x_t\|^2$. 
\end{lemma}
\begin{proof}
Let $\Delta\x =\x' - \x_t$; since $f$ is $\ell$-gradient Lipschitz, $r$ is convex and $\alpha \in (0, 1]$, we have
\begin{align*}
f(\x_t + \alpha\Delta\x_t) - f(\x_t) & \leq \alpha(\Delta\x_t)^\top\nabla f(\x_t) + \frac{\ell\alpha^2\|\Delta\x_t\|^2}{2}, \\
r(\x_t + \alpha\Delta\x_t) - r(\x_t) & \leq \alpha(r(\x_t + \Delta\x_t) - r(\x_t)). 
\end{align*}
We denote $\lambda_t = (\Delta\x_t)^\top\nabla f(\x_t) + r(\x_t + \Delta\x_t) - r(\x_t)$ for the simplicity and derive from the definition of $F$ and $\x_{t+1}$ that
\begin{equation*}
F(\x_{t+1}) - F(\x_t) \leq \alpha\lambda_t + \ell\alpha^2\|\Delta\x_t\|^2/2. 
\end{equation*}
Using the update of $\x'$ and \cite[Proposition~2.4]{Lee-2014-Proximal}, we have $\lambda_t \leq -\Delta\x_t^\top H \Delta\x_t$. Since $H_t \succeq mI$ for some $m > 0$, we have $\lambda_t \leq -m\|\Delta\x_t\|^2$. Putting these pieces together yields that 
\begin{equation*}
F(\x_{t+1}) - F(\x_t) \leq \alpha\lambda_t/4 + (\ell\alpha^2/2 - 3m\alpha/4)\|\Delta\x_t\|^2
\end{equation*}
Therefore, we conclude that $F(\x_{t+1}) - F(\x_t) \leq \alpha\lambda_t/4$ for some $\alpha \in (0, \min\{1, 3m/2\ell\})$. Since $\lambda_t \leq -m\|\Delta\x_t\|^2$ for all $t \geq 0$, we conclude the desired results.  
\end{proof}
\textbf{Proof of Theorem~\ref{theorem:generic-global}:} Using Lemma~\ref{lemma:generic-descent}, we have $F(\x_{t+1}) \leq F(\x_t) - (m\alpha/4)\|\Delta\x_t\|^2$. Then we claim that there exists $\underline{\alpha}>0$ such that $\alpha \geq \underline{\alpha}$ in Algorithm~\ref{Algorithm:PSA}. Indeed, the claim is valid for $\underline{\alpha} = \min\{1/2, 3m/4\ell\} > 0$. This can be shown by the standard arguments for the backtracking line search; see~\cite{Boyd-2004-Convex}. Putting these pieces together yields that
\begin{equation}\label{theorem-inequality-descent-first}
F(\x_{t+1}) \leq F(\x_t) - (m\underline{\alpha}/4)\|\Delta\x_t\|^2. 
\end{equation}
Summing up~\eqref{theorem-inequality-descent-first} over $t = 0, 1, \ldots$ yields that 
\begin{equation*}
0 \leq (m\underline{\alpha}/4)\left(\sum_{t=0}^{+\infty} \|\Delta\x_t\|^2\right) \leq F(\x_0) - F(\x_{T+1}).  
\end{equation*}
This implies that $\sum_{t=0}^\infty \|\Delta\x_t\|^2 \leq (4/m\underline{\alpha})(F(\x_0) - F(\x^*))$. Therefore, we conclude that $\|\Delta\x_t\| \rightarrow 0$ as $t \rightarrow +\infty$. Then it suffices to show that $\x_t$ is optimal if and only if $\|\Delta\x_t\|=0$. Indeed, if $\|\Delta\x_t\|=0$, the optimality of proximal Newton step implies that $-\nabla f(\x_t) \in \partial r(\x_t)$. Thus, $\x_t$ is optimal. Conversely, if $\|\Delta\x_t\| \neq 0$, $\Delta\x_t$ is a descent direction for $F$ at $\x_t$ and $\x_t$ is not optimal. This completes the proof. 

\subsection{Proof of Theorem~\ref{Theorem:local-exact-main}}
We first show that $\alpha=1$ is satisfied by sufficient descent criterion for sufficiently large $t$. Since $f$ is $\rho$-Hessian Lipschitz and $H_t = \nabla^2 f(\x_t)$, we have
{\small \begin{equation*}
f(\x_t + \Delta\x_t) - f(\x_t) \leq (\Delta\x_t)^\top\nabla f(\x_t) + \frac{(\Delta\x_t)^\top H_t\Delta\x_t}{2} + \frac{\rho\|\Delta\x_t\|^3}{6}.  
\end{equation*}}
We denote $\lambda_t = (\Delta\x_t)^\top\nabla f(\x_t) + r(\x_t + \Delta\x_t) - r(\x_t)$ and note $\lambda_t \leq -\Delta\x_t^\top H_t \Delta\x_t$. Putting these pieces together yields that 
\begin{equation*}
F(\x_t + \Delta\x_t) - F(\x_t) \leq \lambda_t/2 + (\rho/6)\|\Delta\x_t\|^3.  
\end{equation*}
For sufficiently large $t$, $\x_t$ is sufficiently close to $\x^*$ such that $\|\x_t - \x^*\| \leq \delta$. From Assumption~\ref{Assumption:regular-main}, $H_t = \nabla^2 f(\x_t) \succeq mI$. This implies that $\lambda_t \leq -m\|\Delta\x_t\|^2$. Therefore, we conclude that 
\begin{equation*}
F(\x_t + \Delta\x_t) - F(\x_t) \leq \lambda_t(1/2 - (\rho/6m)\|\Delta\x_t\|). 
\end{equation*}
Therefore, we have $F(\x_t + \Delta\x_t) - F(\x_t) \leq \lambda_t/4$ if $\|\Delta\x_t\| \leq 3m/2\rho$. Since $\|\Delta\x_t\| \rightarrow 0$, we conclude the desired result.  

In what follows, we assume that $\alpha = 1$ and $\|\x_0 - \x^*\| \leq \{\delta, 2m/3\rho\}$. By the definition of $\x^*$ and $\x_{t+1}$, we have $-\nabla f(\x^*) \in \partial r(\x^*)$ and  
\begin{equation*}
-\nabla f(\x_t) - \nabla^2 f(\x_{t'})(\x_{t+1} - \x_t) \in \partial r(\x_{t+1}). 
\end{equation*}
where $t' \ \text{mod} \ n = 0$ and $0 \leq t - t' \leq n-1$. By the convexity of $r$, we have $(\x_{t+1} - \x^*)^\top(\nabla f(\x^*) - \nabla f(\x_t) - \nabla^2 f(\x_{t'})(\x_{t+1} - \x_t)) \geq 0$. Equivalently, we have
\begin{eqnarray*}
& & (\x_{t+1} - \x^*)^\top(\nabla f(\x^*) - \nabla f(\x_t) - \nabla^2 f(\x_{t'})(\x^* - \x_t)) \\ 
& \geq & (\x_{t+1} - \x^*)^\top\nabla^2 f(\x_{t'})(\x_{t+1} - \x^*). 
\end{eqnarray*}
In what follows, we prove $\|\x_t - \x^*\| \leq \min\{\delta, 2m/3\rho\}$ for all $t \geq 0$ using the induction. It is trivial when $t=0$. Assume that $\|\x_j - \x^*\| \leq \min\{\delta, 2m/3\rho\}$ for all $j \leq t$, Assumption~\ref{Assumption:regular-main} implies that $\nabla^2 f(\x_j) \succeq m I$ for all $j \leq t$. So $(\x_{t+1} - \x^*)^\top\nabla^2 f(\x_{t'})(\x_{t+1} - \x^*) \geq m\|\x_{t+1} - \x^*\|^2$. Putting these pieces together yields that 
\begin{equation}\label{inequality-exact-local-first}
m\|\x_{t+1} - \x^*\| \leq \|\nabla f(\x^*) - \nabla f(\x_t) - \nabla^2 f(\x_{t'})(\x^* - \x_t)\|. 
\end{equation}
Since $f$ is $\rho$-Hessian Lipschitz, we have
\begin{eqnarray}\label{inequality-exact-local-second}
& & \|\nabla f(\x^*) - \nabla f(\x_t) - \nabla^2 f(\x_{t'})(\x^* - \x_t)\| \nonumber \\
& = & \left\|\int_0^1 (\nabla^2 f(\x^* + s(\x_t - \x^*)) - \nabla^2 f(\x_{t'}))(\x_t - \x^*) ds\right\| \nonumber \\
& \leq & \|\x_t - \x^*\|\int_0^1 \|\nabla^2 f(\x^* + s(\x_T - \x^*)) - \nabla^2 f(\x_{t'})\| ds.  \nonumber \\
& \leq & \rho\|\x_t - \x^*\|(\|\x_{t'} - \x^*\| + \|\x_t - \x^*\|/2). 
\end{eqnarray}
Combining Eq.~\eqref{inequality-exact-local-first} and Eq.~\eqref{inequality-exact-local-second}, we have
\begin{equation}\label{inequality-exact-local-third}\small
\|\x_{t+1} - \x^*\| \leq (\rho/m)\|\x_t - \x^*\|(\|\x_{t'} - \x^*\| + \|\x_t - \x^*\|/2). 
\end{equation}
By induction, we have $\|\x_{t+1} - \x^*\| \leq \|\x_t - \x^*\|$ and hence $\|\x_{t+1} - \x^*\| \leq \min\{\delta, 2m/3\rho\}$. By the definition of $r_t$, we conclude that $r_t \leq 1$ for all $t \geq 0$. 

Finally, we show that $r_t \leq r_0^{(n+1)^k(l+1)}$ for $t = nk+l$ where $k \geq 0$ and $0 \leq l \leq n-1$ are integers. Indeed, from the previous analysis, we have $\|\x_t - \x^*\| \leq \|\x_{t'} - \x^*\|$ since $t \geq t'$. This together with Eq.~\eqref{inequality-exact-local-third} implies that 
\begin{equation*}
\|\x_{t+1} - \x^*\| \leq (3\rho/2m)\|\x_t - \x^*\|\|\x_{t'} - \x^*\|. 
\end{equation*}
where $t' \ \text{mod} \ n = 0$ and $0 \leq t - t' \leq n-1$. Equivalently, we have $r_{t+1} \leq r_t r_{t'}$. It suffices to prove the desired result using the induction. Indeed, it holds trivially when $t=0$. Let $t = nk+l$ be an integer for some $k \geq 0$ and $0 \leq l \leq n-1$ such that $r_t \leq r_0^{(l+1)(n+1)^k}$. We consider two cases: (i) if $t = nk$, then $t' = t$ and $r_{t+1} \leq r_t r_{t'} = r_t^2 \leq r_0^{2(n+1)^k}$; (ii) if $t = nk+l$ for $0 < l \leq n-1$, then $t' = nk$ and $r_{t+1} \leq r_t r_{t'} = r_t r_{nk} \leq r_0^{(l+2)(n+1)^k}$. This completes the proof. 

\subsection{Proof of Theorem~\ref{Theorem:local-quasi-main}}
We first show that $\alpha=1$ is satisfied by sufficient descent criterion for sufficiently large $t$ satisfying $t \ \text{mod} \ n = 0$. Using the same argument as in the proof of Theorem~\ref{Theorem:local-exact-main}, we have
\begin{align*}
& F(\x_t + \Delta\x_t) - F(\x_t) \\
\leq & \lambda_t(1/2 - (\rho/6m)\|\Delta\x_t\|) + \Delta\x_t^\top(\nabla^2 f(\x_t) - H_t)\Delta\x_t/2 \\
\leq & \lambda_t(1/2 - (\rho/6m)\|\Delta\x_t\|) + \|\Delta\x_t\|^2\|\nabla^2 f(\x_t) - H_t\|/2. 
\end{align*}
Using the same argument in the proof of Theorem~\ref{theorem:generic-global}, we have $\lambda_t \leq -m\|\Delta\x_t\|^2$. Putting these pieces together yields that 
{\small \begin{equation*}
F(\x_t + \Delta\x_t) - F(\x_t) \leq \lambda_t(1/2 - (\rho/6m)\|\Delta\x_t\| - \|\nabla^2 f(\x_t) - H_t\|/2m).
\end{equation*}}
Since $\x_t$ converges to $\x^*$, $\|\Delta\x_t\| \rightarrow 0$ and $\|H_t - \nabla^2 f(\x_t)\| = o(1)$ holds for all $t \ \text{mod} \ n = 0$, the following inequality holds for sufficiently large $t$: 
\begin{equation*}
(\rho/6m)\|\Delta\x_t\| + \|\nabla^2 f(\x_t) - H_t\|/2m \leq 1/4. 
\end{equation*}
Therefore, $F(\x_t + \Delta\x_t) - F(\x_t) \leq \lambda_t/4$ for sufficiently large $t$ and $t \ \text{mod} \ n = 0$ and we conclude the desired result.

In what follows, we assume that $\x_0$ is sufficiently close to $\x^*$. Indeed, since $\x_0$ is sufficiently close to $\x^*$ and $F(\x_t)$ is a non-increasing sequence, $\{\x \in \br^d: F(\x) \leq F(\x_0)\}$ is contained in $\BB_\delta(\x^*)$ and Assumption~\ref{Assumption:regular-main} implies that $\nabla^2 f(\x_t) \succeq mI_d$ for all $t \geq 0$. By the definition of $\x_{t+1}$, we have
\begin{equation*}
-\nabla f(\x_t) - H_{t'}(\x_{t+1} - \x_t) \in \partial r(\x_{t+1}). 
\end{equation*}
where $t' \ \text{mod} \ n = 0$ and $0 \leq t - t' \leq n-1$. Using the same argument as in the proof of Theorem~\ref{Theorem:local-exact-main}, we derive the analogue of Eq.~\eqref{inequality-exact-local-first} as follows, 
\begin{eqnarray}
m\|\x_{t+1} - \x^*\| & \leq & \|\nabla f(\x^*) - \nabla f(\x_t) - \nabla^2 f(\x_{t'})(\x^* - \x_t)\| \nonumber \\ 
& & + \|(\nabla^2 f(\x_{t'}) - H_{t'})(\x_{t+1} - \x_t)\|. 
\end{eqnarray}
Using Eq.~\eqref{inequality-exact-local-second}, we have
\begin{eqnarray}\label{inequality-quasi-local-first}
\|\x_{t+1} - \x^*\| & \leq & (\rho/m)\|\x_t - \x^*\|(\|\x_{t'} - \x^*\| + \|\x_t - \x^*\|/2)  \nonumber \\
& & + \|(\nabla^2 f(\x_{t'}) - H_{t'})(\x_{t+1} - \x_t)\|/m. 
\end{eqnarray}
Since $\|H_t - \nabla^2 f(\x_t)\| = o(1)$ and $t' \rightarrow +\infty$ as $t \rightarrow +\infty$, we have $\|\nabla^2 f(\x_{t'}) - H_{t'}\| \leq m/2$ after $t$ is sufficiently large. Then
\begin{eqnarray}\label{inequality-quasi-local-second}
& & \|(\nabla^2 f(\x_{t'}) - H_{t'})(\x_{t+1} - \x_t)\|/m \\
& \leq & (\|\nabla^2 f(\x_{t'}) - H_{t'}\|/m)(\|\x_{t+1} - \x^*\| + \|\x_t - \x^*\|) \nonumber \\
& \leq & \|\x_{t+1} - \x^*\|/2 + \|\nabla^2 f(\x_{t'}) - H_{t'}\|\|\x_t - \x^*\|/m. \nonumber. 
\end{eqnarray}
Plugging Eq.~\eqref{inequality-quasi-local-second} into Eq.~\eqref{inequality-quasi-local-first}, we have
\begin{eqnarray*}
\|\x_{t+1} - \x^*\|/2 & \leq & (\rho/m)\|\x_t - \x^*\|(\|\x_{t'} - \x^*\| + \|\x_t - \x^*\|/2)  \nonumber \\
& & + \|\nabla^2 f(\x_{t'}) - H_{t'}\|\|\x_t - \x^*\|/m. 
\end{eqnarray*}
Since $\x_t, \x_{t'} \rightarrow \x^*$ and $\|\nabla^2 f(\x_{t'}) - H_{t'}\| \rightarrow 0$ as $t \rightarrow +\infty$, we have $\|\x_{t+1} - \x^*\| = o(\|\x_t - \x^*\|)$. This completes the proof. 

\subsection{Proof of Theorem~\ref{Theorem:local-inexact-main}}
\begin{lemma}\label{Lemma:inexact-local-grad}
Under Assumption~\ref{Assumption:objective-convex-main}-\ref{Assumption:regular-main} and let $\{\x_t\}_{t\geq 0}$ be generated by Algorithm~\ref{Algorithm:IPSA} with $\|\x_j - \x^*\| \leq \delta$ for all $j \leq t$. Then $(m/2)\|\x_t - \x^*\| \leq \|G(\x_t)\| \leq 2\ell\|\x_t - \x^*\|$. 
\end{lemma}
\begin{proof}
By the definition of $G(\cdot)$ and $\x^*$, we have $-\nabla f(\x^*) \in \partial r(\x^*)$ and $G(\x_t) - \nabla f(\x_t) \in \partial r(\x_t - G(\x_t)/\ell)$. By the convexity of $r$, we have
\begin{equation*}
(\x_t - \x^* - G(\x_t)/\ell)^\top(G(\x_t) - \nabla f(\x_t) + \nabla f(\x^*)) \geq 0.
\end{equation*}
which implies that 
\begin{eqnarray}\label{inequality-inexact-local-grad-first}
& & (\x_t - \x^*)^\top G(\x_t) + G(\x_t)^\top(\nabla f(\x_t) - \nabla f(\x^*))/\ell \nonumber \\ 
& \geq & \|G(\x_t)\|^2/\ell + (\x_t - \x^*)^\top(\nabla f(\x_t) - \nabla f(\x^*)). 
\end{eqnarray}
Since $\|\x_j - \x^*\| \leq \delta$ for all $j \leq t$, Assumption~\ref{Assumption:regular-main} implies that 
\begin{equation}\label{inequality-inexact-local-grad-second}
(\x_t - \x^*)^\top(\nabla f(\x_t) - \nabla f(\x^*)) \geq m\|\x_t - \x^*\|^2. 
\end{equation}
Since $f$ is $\ell$-gradient Lipschitz, we have
\begin{equation}\label{inequality-inexact-local-grad-third}
G(\x_t)^\top(\nabla f(\x_t) - \nabla f(\x^*)) \leq \ell\|\x_t - \x^*\|\|G(\x_t)\|. 
\end{equation}
Using $(\x_t - \x^*)^\top G(\x_t) \leq \|\x_t - \x^*\|\|G(\x_t)\|$ and plugging Eq.~\eqref{inequality-inexact-local-grad-second} and Eq.~\eqref{inequality-inexact-local-grad-third} to Eq.~\eqref{inequality-inexact-local-grad-first} yields that 
\begin{equation*}
\|G(\x_t)\|^2/\ell + m\|\x_t - \x^*\|^2 \leq 2\|\x_t - \x^*\|\|G(\x_t)\|. 
\end{equation*}
which implies the desired result. 
\end{proof}
Let $\bar{\x}_{t+1}$ be defined by an exact proximal Newton step: $\bar{\x}_{t+1} = \argmin\{r(\z) + (\z-\x_t)^\top\nabla f(\x_t) + \frac{1}{2}(\z-\x_t)^\top H_{t'}(\z-\x_t)\}$ where $t' \ \text{mod} \ n = 0$ and $0 \leq t - t' \leq n-1$. 
\begin{lemma}\label{Lemma:inexact-local-resd}
Under Assumption~\ref{Assumption:objective-convex-main}-\ref{Assumption:regular-main} and let $\{\x_t\}_{t\geq 0}$ be generated by Algorithm~\ref{Algorithm:IPSA} with $\|\x_j - \x^*\| \leq \delta$ for all $j \leq t$. Then $\|\x_{t+1} - \bar{\x}_{t+1}\| \leq (2/m)\|\widehat{G}(\x_{t+1}, \x_t, H_{t'})\|$. 
\end{lemma}
\begin{proof}
By the definition of $\widehat{G}(\cdot)$, we have
\begin{eqnarray*}
& & \widehat{G}(\x_{t+1}, \x_t, H_{t'}) - \nabla f(\x_t) - H_{t'}(\x_{t+1} - \x_t) \\
& \in & \partial r(\x_{t+1} - \widehat{G}(\x_{t+1}, \x_t, H_{t'})/\ell).
\end{eqnarray*}
By the definition of $\bar{\x}_{t+1}$, we have $- \nabla f(\x_t) - H_{t'}(\bar{\x}_{t+1} - \x_t) \in \partial r(\bar{\x}_{t+1})$. By the convexity of $r$, we have 
\begin{eqnarray*}
0 & \leq & (\x_{t+1} - \bar{\x}_{t+1} - \widehat{G}(\x_{t+1}, \x_t, H_{t'})/\ell)^\top(\widehat{G}(\x_{t+1}, \x_t, H_{t'}) \\ 
& & - H_{t'}(\x_{t+1} - \bar{\x}_{t+1})). 
\end{eqnarray*}
Equivalently, we have
{\small \begin{align*}
(\x_{t+1} - \bar{\x}_{t+1})^\top H_{t'}(\x_{t+1} - \bar{\x}_{t+1}) & \leq (\x_{t+1} - \bar{\x}_{t+1})^\top\widehat{G}(\x_{t+1}, \x_t, H_{t'}) \\ 
& \hspace*{-12em} + \widehat{G}(\x_{t+1}, \x_t, H_{t'})^\top H_{t'}(\x_{t+1} - \bar{\x}_{t+1})/\ell - \|\widehat{G}(\x_{t+1}, \x_t, H_{t'})\|^2/\ell. 
\end{align*}}
Since $\|\x_j - \x^*\| \leq \delta$ for all $j \leq t$, Assumption~\ref{Assumption:regular-main} implies that $(\x_{t+1} - \bar{\x}_{t+1})^\top H_{t'}(\x_{t+1} - \bar{\x}_{t+1}) \geq m\|\x_{t+1} - \bar{\x}_{t+1}\|^2$. Since $f$ is $\ell$-gradient Lipschitz, we have $H_{t'} \preceq \ell I$ and 
{\small \begin{equation*}
\widehat{G}(\x_{t+1}, \x_t, H_{t'})^\top H_{t'}(\x_{t+1} - \bar{\x}_{t+1}) \leq \ell\|\x_{t+1} - \bar{\x}_{t+1}\|\|\widehat{G}(\x_{t+1}, \x_t, H_{t'})\|.
\end{equation*}}
Putting these pieces with $(\x_{t+1} - \bar{\x}_{t+1})^\top\widehat{G}(\x_{t+1}, \x_t, H_{t'}) \leq \|\x_{t+1} - \bar{\x}_{t+1}\|\|\widehat{G}(\x_{t+1}, \x_t, H_{t'})\|$ yields the desired result. 
\end{proof}
For the ease of presentation, we denote $\theta_1=(6\rho + 2m\ell)/m$, $\theta_2 = \sqrt[\gamma-1]{(3\rho/2m)^{\gamma-1} + (\ell/2)^{\gamma-1}}$ and $\theta_3 = 3\rho/2m + \ell/2$. 
\begin{lemma}\label{Lemma:inexact-local-stay}
Under Assumption~\ref{Assumption:objective-convex-main}-\ref{Assumption:regular-main} and let $\{\x_t\}_{t\geq 0}$ be generated by Algorithm~\ref{Algorithm:IPSA} with $0 < \eta_t \leq \bar{\eta} < m/16\ell$ and $\|\x_0 - \x^*\| \leq \min\{\delta, 1/\theta_1\}$. Then $\|\x_t - \x^*\| \leq \min\{\delta, 1/\theta_1\}$ and $\|\x_{t+1} - \x^*\| \leq (1/2)\|\x_t - \x^*\|$ for all $t \geq 0$. 
\end{lemma}
\begin{proof}
We prove using the induction. It is trivial when $t=0$. Assume $\|\x_j - \x^*\| \leq \min\{\delta, 1/\theta_1\}$ for all $j \leq t$, we hope to show $\|\x_{t+1} - \x^*\| \leq \min\{\delta, 1/\theta_1\}$. 

Since $\bar{\x}_{t+1}$ is achieved by an exact proximal Newton step, it satisfies Eq.~\eqref{inequality-exact-local-third} such that 
\begin{equation*}
\|\bar{\x}_{t+1} - \x^*\| \leq (\rho/m)\|\x_t - \x^*\|(\|\x_{t'} - \x^*\| + \|\x_t - \x^*\|/2). 
\end{equation*}
Using Lemma~\ref{Lemma:inexact-local-grad},~\ref{Lemma:inexact-local-resd} and the stopping criterion~\eqref{criterion:stop-rule}, we have
\begin{align*}
\|\x_{t+1} - \bar{\x}_{t+1}\| & \leq (2/m)\|\widehat{G}(\x_{t+1}, \x_t, H_{t'})\| \leq (2\eta_t/m)\|G(\x_t)\|^\gamma \\
& \leq (2\eta_t/m)(2\ell\|\x_t - \x^*\|)^\gamma.   
\end{align*}
Putting these pieces together yields that 
\begin{align}\label{inequality-inexact-local-stay-first}
\|\x_{t+1} - \x^*\| \leq & (\rho/m)\|\x_t - \x^*\|(\|\x_{t'} - \x^*\| + \|\x_t - \x^*\|/2) \nonumber \\
& + (2\eta_t/m)(2\ell\|\x_t - \x^*\|)^\gamma.
\end{align}
Using the induction that $\|\x_j - \x^*\| \leq \min\{\delta, 1/\theta_1\}$ for all $j \leq t$ and the definition of $\theta_1$, we have $\|\x_{t'} - \x^*\| + \|\x_t - \x^*\|/2 \leq m/4\rho$ and $2\ell\|\x_t - \x^*\| \leq 1$. Putting these pieces together with Eq.~\eqref{inequality-inexact-local-stay-first} and $\gamma \geq 1$ yields that 
\begin{equation}\label{inequality-inexact-local-stay-second}
\|\x_{t+1} - \x^*\| \leq (1/4 + 4\ell\eta_t/m)\|\x_t - \x^*\|. 
\end{equation}
In addition, $0 < \eta_t \leq \bar{\eta} < m/16\ell$. Thus, we have $\|\x_{t+1} - \x^*\| \leq \|\x_t - \x^*\| \leq \min\{\delta, 1/\theta_1\}$. Then we proceed to show that $\|\x_{t+1} - \x^*\| \leq (1/2)\|\x_t - \x^*\|$ for all $t \geq 0$. Indeed, since $\|\x_t - \x^*\| \leq \min\{\delta, 1/\theta_1\}$ for all $t \geq 0$, the desired result follows from Eq.~\eqref{inequality-inexact-local-stay-second} and $0 < \eta_t \leq \bar{\eta} < m/16\ell$. 
\end{proof}
\textbf{Proof of Theorem~\ref{Theorem:local-inexact-main}:} First, we show $r_t \leq 1$ for all $t \geq 0$. Indeed, we have $\x_0$ satisfies Eq.~\eqref{condition:local-inexact-main}. Thus, $\|\x_0 - \x^*\| \leq \min\{\delta, 1/\theta_1\}$ for all $\gamma \geq 1$. In addition, $0 < \eta_t \leq \bar{\eta} < m/16\ell$. Lemma~\ref{Lemma:inexact-local-stay} implies $\|\x_t - \x^*\| \leq \|\x_0 - \x^*\|$ for all $t \geq 0$. By the definition of $r_t$ (cf. Eq~\eqref{def:local-inexact-main}), we have  
\begin{equation*}
r_t \ \leq \ \left\{\begin{array}{cl}
\theta_1\|\x_0 - \x^*\|, & \gamma = 1, \\
\theta_2\|\x_0 - \x^*\|, & \gamma \in (1, 2), \\
\theta_3\|\x_0 - \x^*\|, & \gamma > 2. 
\end{array}\right. 
\end{equation*}
Therefore, we conclude that $r_t \leq 1$ for all $t \geq 0$. Then we prove the remaining parts case by case.  

\textbf{Case I.} If $\gamma = 1$ and $\eta_t \nrightarrow 0$, we have $r_t = \theta_1\|\x_t - \x^*\|$. Lemma~\ref{Lemma:inexact-local-stay} implies that $r_{t+1} \leq r_t/2$ for all $t \geq 0$. Furthermore, Eq.~\eqref{inequality-inexact-local-stay-first} with $\gamma=1$ implies that 
\begin{equation*}
r_{t+1}/r_t \leq (\rho/m) (\|\x_{t'} - \x^*\| + \|\x_t - \x^*\|/2) + 4\eta_t\ell/m. 
\end{equation*}
Note that $t' \ \text{mod} \ n = 0$ and $0 \leq t - t' \leq n-1$, we have $\|\x_{t'} - \x^*\| \rightarrow 0$ and $\|\x_t - \x^*\| \rightarrow 0$ as $t \rightarrow +\infty$. If $\eta_t \rightarrow 0$, we have $r_{t+1}/r_t \rightarrow 0$ and hence $r_{t+1} = o(r_t)$. 

\textbf{Case II.} If $\gamma \in (1, 2)$ and $\eta_t \nrightarrow 0$, we have $r_t = \theta_2\|\x_t - \x^*\|$. Lemma~\ref{Lemma:inexact-local-stay} implies $\|\x_t - \x^*\| \leq \min\{\delta, 1/\theta_1\}$ for all $t \geq 0$. Recall Eq.~\eqref{inequality-inexact-local-stay-first} as follows, 
\begin{align*}
\|\x_{t+1} - \x^*\| \leq & (\rho/m)\|\x_t - \x^*\|(\|\x_{t'} - \x^*\| + \|\x_t - \x^*\|/2) \\
& + (2^{1+\gamma}\ell^\gamma\eta_t/m)\|\x_t - \x^*\|^\gamma.
\end{align*}
Since $\gamma \in (1, 2)$ and $0 < \eta_t \leq \bar{\eta} < m/16\ell$, we have
\begin{align}\label{inequality-inexact-local-first}
\|\x_{t+1} - \x^*\| \leq & (\rho/m)\|\x_t - \x^*\|(\|\x_{t'} - \x^*\| + \|\x_t - \x^*\|/2) \nonumber \\
& + (\ell/2)^{\gamma-1}\|\x_t - \x^*\|^\gamma.
\end{align}
Since $t' \ \text{mod} \ n = 0$ and $0 \leq t - t' \leq n-1$, we have $t' \leq t$. Lemma~\ref{Lemma:inexact-local-stay} implies that $\|\x_t - \x^*\| \leq \|\x_{t'} - \x^*\|$. Putting these pieces with Eq.~\eqref{inequality-inexact-local-first} yields that 
\begin{eqnarray*}
\|\x_{t+1} - \x^*\| & \leq & (3\rho/2m)\|\x_t - \x^*\|\|\x_{t'} - \x^*\| \\
& & \hspace*{-2em} + (\ell/2)^{\gamma-1}\|\x_t - \x^*\|^{\gamma-1}\|\x_{t'} - \x^*\|.
\end{eqnarray*}
Since $\|\x_t - \x^*\| \leq \min\{\delta, 1/\theta_1\} < 2m/3\rho$ and $\gamma \in (1, 2)$, we have $(3\rho/2m)\|\x_t - \x^*\| \leq ((3\rho/2m)\|\x_t - \x^*\|)^{\gamma-1}$. Putting these pieces together with the definition of $\theta_2$ yields that
\begin{equation*}
\|\x_{t+1} - \x^*\| \leq \theta_2^{\gamma-1}\|\x_t - \x^*\|^{\gamma-1}\|\x_{t'} - \x^*\|.  
\end{equation*}
By the definition of $r_t$, we have $r_{t+1} \leq r_t r_{t'}^{\gamma-1}$. It suffices to prove the desired result using induction. Indeed, it holds trivially when $t=0$. Let $t = nk+l$ be an integer for some $k \geq 0$ and $0 \leq l \leq n-1$ such that $r_t \leq r_0^{(l(\gamma-1)+1)(n(\gamma-1)+1)^k}$. We consider two cases: (i) if $t = nk$, then $t' = t$ and $r_{t+1} \leq r_t r_{t'}^{\gamma-1} = r_t^\gamma \leq r_0^{\gamma(n(\gamma-1)+1)^k}$; (ii) if $t = nk+l$ for $0 < l \leq n-1$, then $t' = nk$ and $r_{t+1} \leq r_t r_{t'}^{\gamma-1} = r_t r_{nk}^{\gamma-1} \leq r_0^{((l+1)(\gamma-1)+1)(n(\gamma-1)+1)^k}$. This completes the inductive argument. If $\eta_t \rightarrow 0$, we have $r_{t+1}/(r_t r_{t'}^{\gamma-1}) \rightarrow 0$. Using the same argument as before, we have $r_t/r_0^{(l(\gamma-1)+1)(n(\gamma-1)+1)^k} \rightarrow 0$ for $t = nk+l$ where $k \geq 0$ and $0 \leq l \leq n-1$ are integers.

\textbf{Case III.} If $\gamma > 2$ and $\eta_t \nrightarrow 0$, we have $r_t = \theta_3\|\x_t - \x^*\|$. Lemma~\ref{Lemma:inexact-local-stay} implies $\|\x_t - \x^*\| \leq \min\{\delta, 1/\theta_1\}$ for all $t \geq 0$. By the definition of $\theta_1$, we have $\|\x_t - \x^*\| \leq 1/2\ell$. Since $\gamma > 2$, we have $(2\ell\|\x_t - \x^*\|)^\gamma \leq 4\ell^2\|\x_t - \x^*\|^2$. Putting these pieces with Eq.~\eqref{inequality-inexact-local-stay-first} and $0 < \eta_t \leq \bar{\eta} < m/16\ell$ yields that 
\begin{align}\label{inequality-inexact-local-second}
\|\x_{t+1} - \x^*\| \leq & (\rho/m)\|\x_t - \x^*\|(\|\x_{t'} - \x^*\| + \|\x_t - \x^*\|/2) \nonumber \\
& + (\ell/2)\|\x_t - \x^*\|^2.
\end{align}
Since $t' \ \text{mod} \ n = 0$ and $0 \leq t - t' \leq n-1$, we have $t' \leq t$. Lemma~\ref{Lemma:inexact-local-stay} implies that $\|\x_t - \x^*\| \leq \|\x_{t'} - \x^*\|$. Putting these pieces with Eq.~\eqref{inequality-inexact-local-second} yields that $\|\x_{t+1} - \x^*\| \leq \theta_3\|\x_t - \x^*\|\|\x_{t'} - \x^*\|$. By the definition of $r_t$, we have $r_{t+1} \leq r_t r_{t'}^{\gamma-1}$. Using the same induction argument in the proof of Theorem~\ref{Theorem:local-exact-main}, we conclude the desired result.

\subsection{Proof of Theorem~\ref{Theorem:inexact-step-main}}
Since $\x_0$ satisfies Eq.~\eqref{condition:local-inexact-main} and $0 < \eta_t \leq \bar{\eta} < m/16\ell$, Lemma~\ref{Lemma:inexact-local-stay} is valid and implies $\|\x_t - \x^*\| \leq \min\{\delta, 1/\theta_1\}$ and $\|\x_{t+1} - \x^*\| \leq (1/2)\|\x_t - \x^*\|$ for all $t \geq 0$. From Lemma~\ref{Lemma:inexact-local-grad}, we have
\begin{equation}\label{inequality-step-size-first}
\|G(\x_{t-1})\| \geq (m/2)\|\x_{t-1} - \x^*\|. 
\end{equation}
By the definition of $\widehat{G}$ and $G$, and the nonexpansiveness of a proximal operator $\prox_{r/\ell}$~\cite{Parikh-2014-Proximal}, we have
{\small \begin{align*}
\|\widehat{G}(\x_t, \x_{t-1}, H) - G(\x_t)\| & \leq \|\nabla f(\x_t) - \nabla f(\x_{t-1}) - H(\x_t - \x_{t-1})\| \\
& \hspace*{-6em} \leq \|\nabla f(\x_t) - \nabla f(\x_{t-1}) - \nabla^2 f(\x_{t-1})\left(\x_t - \x_{t-1}\right)\| \\
& \hspace*{-6em} \quad + \|\nabla^2 f(\x_{t-1}) - H\|\|\x_t - \x_{t-1}\|. 
\end{align*}}
Since $H$ is used for updating $\x_{t+1}$, then $H = \nabla^2 f(\x_{t'})$ where $t' \ \text{mod} \ n = 0$ and $0 \leq t - t' \leq n-1$. In addition, $f$ is $\rho$-Hessian Lipschitz, we have $\|\nabla^2 f(\x_{t-1}) - H\| \leq \rho\|\x_{t-1} - \x_{t'}\|$ and 
{\small \begin{equation*}
\|\nabla f(\x_t) - \nabla f(\x_{t-1}) - \nabla^2 f(\x_{t-1})(\x_t - \x_{t-1})\| \leq (\rho/2)\|\x_t - \x_{t-1}\|^2. 
\end{equation*}}
Putting these pieces together yields that  
\begin{align}\label{inequality-step-size-second}
\|\widehat{G}(\x_t, \x_{t-1}, H) - G(\x_t)\| \leq & (\rho/2)\|\x_t - \x_{t-1}\|^2 \\ 
& \hspace*{-6em} + \rho\|\x_{t-1} - \x_{t'}\|\|\x_t - \x_{t-1}\|. \nonumber 
\end{align}
Combining Eq.~\eqref{inequality-step-size-first} and Eq.~\eqref{inequality-step-size-second} yields that
\begin{equation*}
\eta_t \leq \frac{\rho\|\x_t - \x_{t-1}\|^2 + 2\rho\|\x_{t-1} - \x_{t'}\|\|\x_t - \x_{t-1}\|}{m\|\x_{t-1} - \x^*\|}. 
\end{equation*}
Since $\|\x_{t+1} - \x^*\| \leq (1/2)\|\x_t - \x^*\|$ for all $t \geq 0$, we have $\|\x_t - \x_{t-1}\| \leq (3/2)\|\x_{t-1} - \x^*\|$. Therefore, we conclude that 
\begin{equation*}
\eta_t \leq \frac{3\rho\|\x_t - \x_{t-1}\| + 6\rho\|\x_{t-1} - \x_{t'}\|}{2m}. 
\end{equation*}
Since $t' \ \text{mod} \ n = 0$ and $0 \leq t - t' \leq n-1$, we have $t' \rightarrow +\infty$ as $t \rightarrow +\infty$. This implies that $\x_{t'}, \x_t \rightarrow \x^*$ as $t \rightarrow +\infty$. Therefore, we conclude that $\eta_t \rightarrow 0$ as $t \rightarrow +\infty$. 
 


\bibliographystyle{plain}
\bibliography{IEEEabrv,ref}

\end{document}